\documentclass[12pt]{amsart}
\usepackage{amsmath}
\usepackage[numbers,sort&compress]{natbib}

\newtheorem{theorem}{Theorem}[section]
\newtheorem{lemma}[theorem]{Lemma}

 \theoremstyle{definition}

\theoremstyle{remark}
\newtheorem{remark}[theorem]{Remark}
\allowdisplaybreaks[4]

\numberwithin{equation}{section}

\makeatletter
\@namedef{subjclassname@2020}{\textup{2020} Mathematics Subject Classification}
\makeatother

\usepackage[numbers,sort&compress]{natbib}
\begin{document}

\title[sum Hessian quotient equations]
{A class of  sum Hessian quotient equations arising in conformal geometry}

\author{Ni Xiang}
\address{Faculty of Mathematics and Statistics, Hubei Key Laboratory of Applied Mathematics, Hubei University,  Wuhan 430062, P.R. China}
\email{nixiang@hubu.edu.cn}

\author{Yuni Xiong$^\ast$}
\address{Faculty of Mathematics and Statistics, Hubei Key Laboratory of Applied Mathematics, Hubei University,  Wuhan 430062, P.R. China}
\email{202121104011767@stu.hubu.edu.cn}

\keywords{Sum Hessian quotient equations; The conformal geometry.}

\subjclass[2020]{35J60; 58J05.}
\thanks{This research was supported by funds from the National Natural Science Foundation of China No.12571213, No. 12426532.}
\thanks{$\ast$ Corresponding author}

\begin{abstract}
In this paper, we investigate the existence result of $(n-1)$-admissible solutions to a class of sum Hessian quotient equations with the modified Schouten tensor on Riemannian manifolds.
By establishing uniform a priori estimates, we prove the solvability of the problem via degree theory.
\end{abstract}

\maketitle
%%%%%%%%%%%%%%%%%%%%%%%%%%%%%%%%%%%
\section{Introduction}
Let $\left(\overline{M}^n,g\right)$ be a smooth closed Riemannian manifold of dimension $n\geq 3$ with smooth boundary $\partial M$, and $M=\overline{M}\setminus \partial M$ be the interior of $\overline{M}$. In the study of conformal geometry, the modified Schouten tensor is defined by
$$A^{\tau}_{\widetilde{g}}={\frac{1}{{n-2}}} \left({Ric}_{{\widetilde{g}}}-\frac{\tau R_{\widetilde{g}}}{2(n-1)}\widetilde{g} \right),$$
where $\widetilde{g}\in [g]$ belongs to the conformal class of $g$,  ${Ric}_{\widetilde{g}}$ and ${R}_{\widetilde{g}}$ are the Ricci and scalar curvatures of $\widetilde{g}$ respectively. This tensor was independently introduced by Gursky-Viaclovsky \cite{Ger03} and Li-Li \cite{LL03} to explore the conformal deformation of curvatures.
In this paper, we investigate the following sum Hessian quotient equation involving the modified Schouten tensor $A_{\widetilde{g}}^\tau$
\begin{align}\label{eq121501}
G(A_{\widetilde{g}}^\tau)=\frac{S_n(A_{\widetilde{g}}^\tau)}{S_{n-1}(A_{\widetilde{g}}^\tau)}
=\frac{\sigma_n(A_{\widetilde{g}}^\tau)+\alpha(x)\sigma_{n-1}(A_{\widetilde{g}}^\tau)}{\sigma_{n-1}(A_{\widetilde{g}}^\tau)+\alpha(x)\sigma_{n-2}(A_{\widetilde{g}}^\tau)}=f(x),
\end{align}
where $f(x)>0$ and $\alpha(x)\geq 0$ are smooth functions on $M$.
The $k$-th elementary symmetric function is defined by
\begin{equation*}
\sigma_k(\lambda)= \sum _{1 \le i_1 < i_2 <\cdots<i_k\leq
n}\lambda_{i_1}\lambda_{i_2}\cdots\lambda_{i_k},
\end{equation*}
and $\sigma_k(A^{\tau}_{\widetilde{g}})$ means that $k$-elementary symmetric polynomial $\sigma_k$ is applied to the eigenvalues of $g^{-1}\cdot A^{\tau}_{\widetilde{g}}$.

This operator naturally generalizes the classic Hessian operator $\sigma_k$ through a linear combination of Hessian operators. The study of such non-homogeneous equations is deeply rooted in both convex geometry and theoretical physics.
A prominent  example is the Fu-Yau equation
$$\sigma_1(i\partial\overline{\partial}(e^u+\alpha'e^{-u}))+\alpha'\sigma_2(i\partial\overline{\partial}u)=\phi$$
on $n$ dimensional compact K\"{a}hler manifolds. This equation originates from the study of the Hull-Strominger system in  superstring theory and has been extensively investigated in recent years, with  fundamental contributions by Fu-Yau \cite{Fy08,Fy07} and Phong-Picard-Zhang \cite{PPZ19,PPZ17,PPZ}.
Another  significant instance is the special Lagrangian equation introduced by Harvey-Lawson \cite{HL82}, which can be expressed as an alternative combinations of elementary symmetric functions
\begin{equation*}
\sin\Theta\sum_{k=0}^{n/2}(-1)^k\sigma_{2k}(D^2u)+\cos\Theta\sum_{k=0}^{(n-1)/2}(-1)^k\sigma_{2k+1}(D^2u)=0.
\end{equation*}
This equation plays a fundamental role in calibrated geometry and the theory of minimal submanifolds.
Furthermore, equations of the form
\begin{align}
\sigma_1(D^2u)+b\sigma_n(D^2u)=C,
\end{align}
where $b\geq0$ and $C>0$ are constants, arise naturally from the study of the $J$-equation on toric varieties by Collins-Sz\'{e}kelyhidi \cite{CS17}.

The well-known $\sigma_k$-Yamabe problem is of finding a conformal metric $\widetilde{g}\in [g]$ satisfies
$$\sigma_k(A_{\widetilde{g}})=constant.$$
This problem and its variants have been extensively studied with significant contributions in \cite{Ya60,Au76,Sc84,Via00-,Via02,STW07,CGY02-1, CGY02-2,GW06,TW10,BV04,GV07,GW3,SZ06,LS}. Gursky-Viaclovsky \cite{Ger03} investigated a conformal metric $\widetilde{g}\in [g]$ satisfying
\begin{equation}\label{eq051202}
\sigma_k\left(A^{\tau}_{\widetilde{g}}\right)=f(x)
\end{equation}
for any smooth function $f(x)<0$ with $\lambda(-A^{\tau}_{\widetilde{g}})\in \Gamma_k$, $\tau<1$. A parabolic proof was later given by Li-Sheng \cite{LS05}. Recent work has gradually focused on Krylov type equations and sum Hessian equations, both of which involve linear combinations of elementary symmetric functions. The Krylov type equation with the modified Schouten tensor, is given by
\begin{equation}\label{eq051203}
\sigma_k\left(-A_{\tilde{g}}^\tau\right)+\alpha(x)\sigma_{k-1}\left(-A_{\tilde{g}}^\tau\right)=\sum_{l=0}^{k-2}\alpha^l(x)\sigma_l\left(-A_{\tilde{g}}^\tau\right).
\end{equation}
Under the conformal change of metric $\widetilde{g}=e^{2u}g$, the modified Schouten tensor transforms as
\begin{equation}\label{eq051401}
A_{\widetilde{g}}^{\tau}=\frac{\tau-1}{n-2}(\Delta u)g-\nabla^2u+du\otimes du+\frac{\tau-2}{2}|\nabla u|^2g+A_{g}^\tau.
\end{equation}
Chen-He \cite{CH24} proved the existence of a conformal metric $\widetilde{g}\in [g]$ satisfying equation \eqref{eq051203},  provided that $\lambda(-A_{\tilde{g}}^\tau)\in \Gamma_k$, $\tau<1$ and $\alpha^l(x)>0$ for $0\leq l\leq k-2$.
Furthermore, they specifically investigated the case where $\tau=0$, associated with the prescription of negative Ricci curvature.
%They established the existence for the following equation with $\lambda(-Ric_{\widetilde{g}})\in \Gamma_k$
%\begin{equation*}
%\sigma_k(-Ric_{\widetilde{g}})=\sum_{l=0}^{k-1}\alpha^l(x)\sigma_l\left(-Ric_{\widetilde{g}}\right),
%\end{equation*}
%where the coefficients $\alpha_0, \cdots, \alpha_{k-2}$ are nonnegative real numbers and $\alpha_{k-1}$ is a real number.
%In a related development, Chen-Guo-He \cite{CGH22} established the existence result to \eqref{eq051203} with $\tau=0$, $\alpha^0>0$, $\alpha^l\equiv0$ for $l\neq0$.
%Motivated by these foundational works, we seek a conformal metric $\widetilde{g}\in[g]$ satisfies the equation \eqref{eq121501} with the modified Schouten tensor in conformal geometry.
%In light of the conformal transformation $\widetilde{g}=e^{2u}g$, we have

%The case $\tau=1$ which recovers the standard Schouten tensor, the existence to equations
%\begin{equation*}
%\sigma_k(A_{\widetilde{g}})+\alpha(x)\sigma_{k-1}(A_{\widetilde{g}})=\sum_{l=0}^{k-2}\alpha^l(x)\sigma_l(A_{\widetilde{g}})
%\end{equation*}
%was obtained by Chen-Guo-He \cite{CGH22} with $\alpha^l(x)\geq 0$.
%Then we naturally want to find a metric $\widetilde{g}\in[g]$ under $\widetilde{g}=e^{-2u}g$ satisfies the equation with Schouten tensor in conformal geometry
%\begin{align}\label{eq020502}
%G(A_{\widetilde{g}}^1)=\frac{S_n}{S_{n-1}}\left(\nabla^2u+du\otimes du-\frac{1}{2}|\nabla u|^2g+A_{g}^1\right)=f(x),
%\end{align}
In this paper, we study  more general equations. Let's consider
\begin{equation}\label{eq04221}
  U(u)=\lambda(\Delta u)g+\nabla^{2}u+a(x)du\otimes du+b(x)|\nabla u|^{2}g+B, \quad \lambda>0,
\end{equation}
and
\begin{equation}\label{eq052001}
\widetilde{U}(u)=\lambda(\Delta u)g-\nabla^{2}u+a(x)du\otimes du+b(x)|\nabla u|^{2}g+B, \quad \lambda>1,
\end{equation}
where $B$ is a symmetric 2-tensor, $a(x)$ and  $b(x)$ are smooth functions on $M$.
We only introduce the case of $U(u)$, since the proof for the other case $\widetilde{U}(u)$ is similar.
Under a conformal change of metric $\widetilde{g}=e^{2u}g$, the problem reduces to finding a solution $u$ for the sum Hessian quotient equation
\begin{align}\label{eq050601}
F(U)=\frac{\widetilde{S}_n(U)}{\widetilde{S}_{n-1}(U)}=\frac{\sigma_n(U)+\alpha(x) e^{2u(x)} \sigma_{n-1}(U)}{\sigma_{n-1}(U)+\alpha (x) e^{2u(x)}\sigma_{n-2}(U)}=e^{2u(x)}f(x),
\end{align}
where $\alpha(x)\geq 0$ and $f(x)>0$ are smooth functions on $M$. We say $u$ is a $k$-admissible solution of \eqref{eq050601} if $\lambda (g^{-1}\cdot U)\in \Gamma_k$, where
$$\Gamma_k=\{\lambda\in\mathbb{R}^n|\sigma_m(\lambda)>0,\,m=1,\cdots,k\}.$$
Then we derive the following existence theorem.
\begin{theorem}\label{thm051104}
Let $(M,g)$ be a smooth closed Riemannian manifold of dimension $n\geq 3$. Assume that $\alpha(x)\geq 0$ and $f(x)>0$ are smooth functions on $M$, $\lambda>0$, $B \in \Gamma_n$. Then there exists a conformal metric $\widetilde{g}\in [g]$ satisfies equation \eqref{eq050601} with $\lambda(g^{-1}\cdot U)\in \Gamma_{n-1}$.
\end{theorem}

\begin{remark}
The theorem \ref{thm051104}  covers
the special case $\tau=0$. This case corresponds precisely to the problem of prescribing negative Ricci curvature in conformal geometry,
which is previously studied in  sum Hessian quotient equations
\begin{align*}\label{eq020501}
G\left(-A_{\widetilde{g}}^0\right)=\frac{S_n}{S_{n-1}}\left(\frac{-Ric_{\widetilde{g}}}{n-2}\right)=f(x),
\end{align*}
associated with
\begin{equation*}\label{eq051402}
\frac{-Ric_{\widetilde{g}}}{n-2}=\frac{1}{n-2}(\Delta u)g+\nabla^{2}u-du\otimes du+|\nabla u|^{2}g+\frac{-Ric_{g}}{n-2}.
\end{equation*}
%can be recovered by setting the parameters $\lambda=\frac{1}{n-2}$, $a=-1$, $b=1$, $B=\frac{-Ric_{g}}{n-2}$.
\end{remark}

\begin{remark}
In addition to the case treated in Theorem \ref{thm051104},
analogous existence results hold for sum Hessian quotient equations of the form
\begin{align*}
F(\widetilde{U})=\frac{\widetilde{S}_n(\widetilde{U})}{\widetilde{S}_{n-1}(\widetilde{U})}=e^{2u}f(x),
\end{align*}
where
\begin{equation*}
\widetilde{U}(u)=\lambda(\Delta u)g-\nabla^{2}u+a(x)du\otimes du+b(x)|\nabla u|^{2}g+B
\end{equation*}
under conditions $\lambda>1$, $B \in \Gamma_n$, $\lambda(g^{-1}\cdot \widetilde{U})\in \Gamma_{n-1}$. The proof is analogous to that of Theorem \ref{thm051104} and is therefore omitted.
\end{remark}
The rest of the paper is organized as follows. Section 2 is devoted to some preliminary results. In Section 3, we first establish  the uniform $C^0$, $C^1$ and $C^2$ estimates for solutions. Then the proof of Theorem \ref{thm051104} is presented  by applying standard degree theory.

%\begin{theorem}\label{thm122401}
%Let $(M,g_0)$ be a smooth closed Riemannian manifold of dimension $n\geq 3$, $\lambda(-Ric_{g_0})\in \Gamma_n$ be the eigenvalues of $g_0^{-1}\cdot (-Ric_{g_0})$.
%Assume $\alpha$ and $f$ be two positive functions on $M$. Then there is a conformal metric $g\in [g_0]$ satisfies equation \eqref{eq020501}.
%\end{theorem}

%\begin{remark}
%The theorem can be generalized to equations of the more general form \eqref{eq121501}.
%\end{remark}

%%%%%%%%%%%%%%%%%%%%%%%%%%%%%%%%%%%%%%%%%%%%%%%%%%%%%%%%%%%%%%%%%%%%%%%%%%%%%%%%%%%%%%%%%%%%%%%%%%%%%%%%%%%%%%%%%%%%%%%%%%%%%%%%%%%%%%%%%%%%%%%%%%%%%%%%%%%%%%%%%%%%
\section{Preliminaries}
In this section, we present some basic formulas and lemmas for the subsequent analysis, which can be found in \cite{LRW19,R,RW25a,RW25b}. We first state several basic properties of sum Hessian operators $\widetilde{S}_k(\lambda(U))=\sigma_k(\lambda(U))+\alpha(x) e^{2u(x)} \sigma_{k-1}(\lambda(U))$.
\begin{lemma}\label{lem2.5}
Suppose that $\lambda=(\lambda_1,\cdots,\lambda_n)\in \Gamma_{k-1}$, $1\leq
l<k\leq n$, then
\begin{equation}\label{e2.4}
\frac{\sigma_{k-1}(\lambda)\sigma_l(\lambda)}{C_n^{k-1}C_n^l}\geq\frac{\sigma_k(\lambda)\sigma_{l-1}(\lambda)}{C_n^kC_n^{l-1}},
\end{equation}
\end{lemma}
\begin{proof}
If $\sigma_k(\lambda)>0$, inequality \eqref{e2.4} follows directly from the Newton-MacLaurin inequality. If $\sigma_k(\lambda) \leq 0$, the result obviously holds.
\end{proof}

\begin{lemma}\label{le050801}
Let $\lambda=(\lambda_1,\cdots,\lambda_n)\in \mathbb{R}^n$ and $k=1,2,\cdots,n$, then
\par
(1) $\widetilde{S}_k^{pp}(\lambda):=\frac{\partial
\widetilde{S}_k(\lambda)}{\partial\lambda_p}=\sigma_{k-1}(\lambda|p)+\alpha e^{2u}\sigma_{k-2}(\lambda|p)=\widetilde{S}_{k-1}(\lambda|p)$;
\par
(2) $\widetilde{S}_k^{pp,qq}(\lambda):=\frac{\partial^2\widetilde{S}_k(\lambda)}{\partial\lambda_p\partial\lambda_q}=\widetilde{S}_{k-2}(\lambda|pq)$;
\par
(3) $\widetilde{S}_k(\lambda)=\lambda_i \widetilde{S}_{k-1}(\lambda|i)+\widetilde{S}_k(\lambda|i)$;
\par
(4) $\sum_i \widetilde{S}_k(\lambda|i)=(n-k)\widetilde{S}_k(\lambda)+\alpha e^{2u}\sigma_{k-1}(\lambda)$;
\par
(5) $\sum_i\lambda_i\widetilde{S}_{k-1}(\lambda|i)=k\widetilde{S}_k(\lambda)-\alpha e^{2u}\sigma_{k-1}(\lambda)$.
\end{lemma}
\begin{proof}
The proof can be seen in \cite{RW25a} and \cite{LRW19}.
\end{proof}

To solve equation \eqref{eq050601}, we consider a family of equations
\begin{align}\label{eq050701}
F(U^t)=\frac{\widetilde{S}_n(U^t)}{\widetilde{S}_{n-1}(U^t)}=\frac{\sigma_n(U^t)+\alpha e^{2u} \sigma_{n-1}(U^t)}{\sigma_{n-1}(U^t)\
+\alpha e^{2u} \sigma_{n-2}(U^t)}=e^{2u}f^t(x),
\end{align}
where $t\in [0,1]$, $B_g=tB+(1-t)g$ and
$$U^t(u)=\lambda(\Delta u)g+\nabla^{2}u+a(x)du\otimes du+b(x)|\nabla u|^{2}g+B_g,$$
$$f^t(x)=tf(x)+(1-t)\frac{\widetilde{S}_n(B_g)}{\widetilde{S}_{n-1}(B_g)}.$$
%The following lemma says the proper admissible set of \eqref{eq050701} is $\Gamma_{n-1}$ which was first observed by Li-Ren-Wang \cite{LRW19}.
Furthermore, as established by Li-Ren-Wang \cite{LRW19}, the natural admissible set for equation \eqref{eq050701} is the cone $\Gamma_{n-1}$. Then we obtain the following lemma and  omit the proof here.
\begin{lemma}\label{le050901}
The operator $F(U^t)$ in \eqref{eq050701}
%\begin{align*}
%F(U^t)=\frac{\widetilde{S}_n(U^t)}{\widetilde{S}_{n-1}(U^t)}=\frac{\sigma_n(U^t)+\alpha e^{2u}\sigma_{n-1}(U^t)}{\sigma_{n-1}(U^t)\
%+\alpha e^{2u}\sigma_{n-2}(U^t)}
%\end{align*}
is elliptic and concave if the eigenvalues of $g^{-1}\cdot U^t$ belong to $\Gamma_{n-1}$. And for any $n\times n$ symmetric matrix $\xi=(\xi_1,\cdots,\xi_n)$, we have
$$\sum_{i,j,r,s}F^{ij,rs}\xi_{ij}\xi_{rs} \leq 0,$$
where $F^{ij}=\frac{\partial F(U^t)}{\partial (U^t)_{ij}}$, $F^{ij,rs}=\frac{\partial^2F(U^t)}{\partial(U^t)_{ij}\partial(U^t)_{rs}}$.
\end{lemma}

Next, we present some results to handle the troublesome term that arises from the lack of homogeneity.

\begin{lemma}\label{le051201}
Let  $\eta=(\eta_1,\dots,\eta_n)$ be eigenvalues of $g^{-1}\cdot U^t$. If $\eta \in \Gamma_{n-1}$, $f(x)>0$ and $\alpha(x)\geq 0$ are smooth functions, then \\
(1) $\frac{1}{n}\leq \sum_i F^{ii}\leq 2 $;\\
(2) $\sum_{i}F^{ii}\eta_{i}\leq 2e^{2u}f^t$.
\end{lemma}
\begin{proof}
(1) From  Lemma \ref{le050801}
\begin{align*}
\sum_i F^{ii}=&\,\frac{1}{\widetilde{S}^2_{n-1}}[\widetilde{S}_{n-1}(\widetilde{S}_{n-1}+\widetilde{\alpha}\sigma_{n-2})
-\widetilde{S}_{n}(2\widetilde{S}_{n-2}+\widetilde{\alpha}\sigma_{n-3})]\\
=&\,\frac{1}{\widetilde{S}^2_{n-1}}[\sigma^2_{n-1}-2\sigma_n\sigma_{n-2}
+\widetilde{\alpha}\sigma_{n-1}\sigma_{n-2}-3\widetilde{\alpha}\sigma_n\sigma_{n-3}\\
&\,+2\widetilde{\alpha}^2\sigma^2_{n-2}-3\widetilde{\alpha}^2\sigma_{n-1}\sigma_{n-3}],
\end{align*}
where $\widetilde{\alpha}=\alpha e^{2u}$. Accord to Lemma \ref{lem2.5}
%\begin{align*}
%&\sigma_{n-1}^2\geq \frac{2n}{n-1}\sigma_n\sigma_{n-2},\\
%&\sigma_{n-1}\sigma_{n-2}\geq \frac{3n}{n-2}\sigma_n\sigma_{n-3},\\
%&\sigma_{n-2}^2\geq \frac{3(n-1)}{2(n-2)}\sigma_{n-1}\sigma_{n-3},
%\end{align*}
%which follows
\begin{align*}
\sum_i F^{ii}
\geq&\, \frac{1}{\widetilde{S}^2_{n-1}}\left[\frac{1}{n}\sigma_{n-1}^2+\frac{2}{n}\widetilde{\alpha}\sigma_{n-1}\sigma_{n-2}+\frac{2}{n-1}\widetilde{\alpha}^2\sigma_{n-2}^2\right]\\
=&\, \frac{1}{\widetilde{S}^2_{n-1}}\left[\frac{1}{n}\sigma_{n-1}(\sigma_{n-1}+\widetilde{\alpha}\sigma_{n-2})+\widetilde{\alpha}\sigma_{n-2}\left(\frac{1}{n}\sigma_{n-1}+\frac{2}{n-1}\widetilde{\alpha}\sigma_{n-2}\right)\right]\\
\geq &\,
\frac{1}{\widetilde{S}^2_{n-1}}\left[\frac{1}{n}\sigma_{n-1}\widetilde{S}_{n-1}+\frac{1}{n}\widetilde{\alpha}\sigma_{n-2}\widetilde{S}_{n-1}\right]\\
=&\, \frac1n.
\end{align*}
For the upper bound, we have
\begin{align*}
\sum_i F^{ii}=&\,\frac{1}{\widetilde{S}^2_{n-1}}[\widetilde{S}_{n-1}(\widetilde{S}_{n-1}+\widetilde{\alpha}\sigma_{n-2})
-\widetilde{S}_{n}(2\widetilde{S}_{n-2}+\widetilde{\alpha}\sigma_{n-3})]\\
=&\,\frac{1}{\widetilde{S}^2_{n-1}}[\widetilde{S}_{n-1}(2\widetilde{S}_{n-1}-\sigma_{n-1})
-2\widetilde{S}_{n}\widetilde{S}_{n-2}-\widetilde{\alpha}\widetilde{S}_{n}\sigma_{n-3})]\\
\leq&\, 2.
\end{align*}

(2) From equation \eqref{eq050701} and Lemma \ref{le050801}, a straight forward calculation shows
\begin{align*}
\sum_iF^{ii}\eta_{i}&=\frac{\widetilde{S}_{n-1}(n\widetilde{S}_{n}-\widetilde{\alpha} \sigma_{n-1})-\widetilde{S}_{n}[(n-1)\widetilde{S}_{n-1}-\widetilde{\alpha}\sigma_{n-2}]}{\widetilde{S}_{n-1}^{2}}\\
&=e^{2u}f^t+\frac{1}{\widetilde{S}_{n-1}^{2}}[\widetilde{\alpha}\widetilde{S}_{n}\sigma_{n-2}-\widetilde{\alpha}\widetilde{S}_{n-1}\sigma_{n-1}]\\
&\leq e^{2u}f^t+\frac{\widetilde{\alpha}\widetilde{S}_n\sigma_{n-2}}{\widetilde{S}_{n-1}^2}\\
&=e^{2u}f^t+\frac{\widetilde{S}_n(\widetilde{S}_{n-1}-\sigma_{n-1})}{\widetilde{S}_{n-1}^2}\\
&\leq2e^{2u}f^t.
\end{align*}
Therefore, we finish the proof.
\end{proof}

%%%%%%%%%%%%%%%%%%%%%%%%%%%%%%%%%%%%%%%%%%%%%%%%%%%%%%%%%%%%%%%%%%%%%%%%%%%%%%%%%%%%%%%%%%%%%%%%%%%%%%%%%%%%%%%%%%%%%%%%%%%%%%%%%%%%%%%%%%%%%%%%%%%%%%%%%%%%%%%%%%%%
\section{Proof of Estimates}
\subsection{Proof of $C^0$ Estimates}
In the subsection, we establish the following theorem which provides uniform $C^0$ bounds to equation \eqref{eq050701}.
\begin{theorem}\label{thm051101}
%Let $(M,g)$ be a smooth closed Riemannian manifold of dimension $n\geq 3$,
Let $f(x)>0$ and $\alpha(x)\geq 0$ be smooth functions on $M$. If $u$ is a $(n-1)$-admissible solution of \eqref{eq050701} with $\lambda>0$ and $B\in \Gamma_{n}$, then
\begin{equation}
\sup_M |u|\leq C,
\end{equation}
where $C$ depends on $g$, $B$, $\sup \alpha$, $\inf f$, $\|f\|_{C^0}$.
\end{theorem}
\begin{proof}
Assume that the minimum of $u$ is attained at $x_0\in M$, thus
$$ \nabla u(x_0)=0, \quad \nabla^{2}u(x_0)\geq 0,$$
which implies $U^t \geq B_g$ at $x_0$.
By Lemma \ref{le050901}, we show that $F$ is elliptic and  concave in $\Gamma_{n-1}$, then we have the inequality
\begin{equation}
e^{2u(x_0)}\left[tf(x_0)+(1-t)\frac{\widetilde{S}_n(B_g)}{\widetilde{S}_{n-1}(B_g)}\right]=\frac{\widetilde{S}_n(U^t)}{\widetilde{S}_{n-1}(U^t)}\geq \frac{\widetilde{S}_n(B_g)}{\widetilde{S}_{n-1}(B_g)}.
\end{equation}
Without loss of generality, we suppose that $0<e^{2u(x_0)} \leq 1$, otherwise $e^{2u(x_0)} \geq 1$. It yields that
%begin{align*}
%e^{2u(x_0)}f(x_0)\left[\sigma_{n-1}(B_g)+\alpha e^{2u(x_0)} \sigma_{n-2}(B_g)\right]\geq \sigma_n(B_g)+\alpha e^{2u} \sigma_{n-1}(B_g)
%\geq\sigma_n(B_g),
%\end{align*}
\begin{align*}
e^{2u(x_0)}f(x_0)\sigma_{n-1}(B_g)+\alpha e^{2u(x_0)}f(x_0) \sigma_{n-2}(B_g)
\geq\sigma_n(B_g).
\end{align*}
Hence $e^{2u(x_0)}\geq C$, where $C$ is a positive constant depending on $g, B, \max \alpha, \max f$.

Similarly, we suppose the maximum of $u$ is attained at $x_1$ and $e^{2u(x_1)} \geq 1$, thus
$$e^{2u(x_1)}f(x_1)\left[\sigma_{n-1}(B_g)+\alpha e^{2u(x_1)} \sigma_{n-2}(B_g)\right]\leq \sigma_{n}(B_g)+\alpha e^{2u(x_1)} \sigma_{n-1}(B_g),$$
which follows
$$e^{2u(x_1)}f(x_1)\sigma_{n-1}(B_g)+\alpha e^{2u(x_1)}\left[e^{2u(x_1)}f(x_1)\sigma_{n-2}(B_g)-\sigma_{n-1}(B_g)\right]\leq \sigma_{n}(B_g).$$
If $e^{2u(x_1)}f(x_1)\sigma_{n-2}(B_g)-\sigma_{n-1}(B_g)\leq 0$, we obtain $e^{2u(x_1)} \leq C$. Otherwise
$$e^{2u(x_1)}f(x_1)\sigma_{n-1}(B_g)\leq \sigma_{n}(B_g),$$
which implies $e^{2u(x_1)}\leq C$, where $C$ is a positive constant depending on $ g, B, \inf f$.
\end{proof}

Next, we present the following uniqueness result.
\begin{theorem}\label{thm051102}
$u=0$ is the unique solution for
\begin{align}
F(U^0)=\frac{\widetilde{S}_n(U^0)}{\widetilde{S}_{n-1}(U^0)}=\frac{\sigma_n(U^0)+\alpha e^{2u} \sigma_{n-1}(U^0)}{\sigma_{n-1}(U^0)\
+\alpha e^{2u} \sigma_{n-2}(U^0)}=e^{2u}f^0(x).
\end{align}
\end{theorem}
\begin{proof}
If the minimum of $u$ is attained at $x_0$, thus $\nabla u(x_0)=0$, $\nabla^2 u(x_0)\geq0$, which means
$U^0\geq g$ at the point $x_0$. From Lemma \ref{le050901}, we get
$$e^{2u(x_0)}\frac{\widetilde{S}_n(g)}{\widetilde{S}_{n-1}(g)}=\frac{\widetilde{S}_n(U^0)}{\widetilde{S}_{n-1}(U^0)}\geq \frac{\widetilde{S}_n(g)}{\widetilde{S}_{n-1}(g)},$$
which yields $e^{2u(x_0)}\geq 1$. On the other hand, we assume that the maximum of $u$ is attained at $x_1$, then $e^{2u(x_1)}\leq 1$. Consequently, we finish the proof.
\end{proof}

\subsection{Proof of $C^1$ estimates }
We show  $C^{1}$ estimates for sum Hessian quotient equations \eqref{eq050701} in this subsection.

\begin{theorem}\label{thm050701}
Let $f(x)>0$ and $\alpha(x)\geq 0$ be smooth functions on $M$.
If $u\in C^{3}(B_{r})$ is a $(n-1)$-admissible solution of equation \eqref{eq050701} in $B_{r}\subset M$ with $\lambda>0$ and $B\in \Gamma_{n}$, then
\begin{equation*}
  \sup_{B_{r/2}}|\nabla u|\leq C,
\end{equation*}
where $C$ is a positive constant depending on $\lambda$, $g$, $\|\alpha\|_{C^{1}}$, $\|a\|_{C^{1}}$, $\|b\|_{C^{1}}$, $\|B\|_{C^{1}}$, $\|f\|_{C^1}$  and $\|u\|_{C^{0}}$.
\end{theorem}
\begin{proof}
Let $\rho(x)\in C^{\infty}_{0}(B_{r})$ be a cut off function satisfies $\rho=1$ in $B_{r/2}$ and $\rho=0$ outside $B_{r}$.
Consider the auxiliary function
$$Q(x)=\rho(x)ve^{\phi(u)},$$
where $v=1+\frac{|\nabla u|^2}{2}$,  $\phi(u)=[\varphi(u)]^{-N}$ and  $\varphi(u)=1+u-\inf_{\{\rho>0\}}u$, $N\gg1$. For convenience, denote
$$\phi'=\frac{d \phi(u)}{d u}, \quad \phi''=\frac{d^2 \phi(u)}{d u^2},$$
thus
\begin{equation*}
  \phi''-\frac{3}{4}(\phi')^{2}%=-\phi'\left(\frac{N+1}{P}+\frac{3}{4}\phi'\right)
  = -\phi'\left(\frac{N+1}{\varphi}-\frac{3N}{4\varphi^{N+1}}\right)\geq -\frac{N\phi'}{4\varphi}%=\frac{N^{2}}{4P^{N+2}}
  >0.
\end{equation*}
Define
$$\widetilde{Q}=\log Q(x)=\log \rho(x)+\log v+\phi(u).$$
Assume that the maximum of $\widetilde{Q}$ is attained at an interior point $x_0 \in B_{r}$ and  $U^t(u(x_0))$ is diagonal. From now on, all calculations are performed at $x_0$.
\begin{equation}\label{eq4.1}
0=\widetilde{Q}_{i}=\frac{\rho_{i}}{\rho}+\frac{v_{i}}{v}+\phi'u_{i},
\end{equation}
and
\begin{equation}\label{eq4.3}
  0\geq \widetilde{F}^{ij}\widetilde{Q}_{ij}
  =\widetilde{F}^{ij}\left(\frac{\rho_{ij}\rho-\rho_{i}\rho_{j}}{\rho^{2}}+\frac{v_{ij}v-v_{i}v_{j}}{v^{2}}+\phi^{''}u_{i}u_{j}+\phi'u_{ij}\right),
\end{equation}
where
$$\widetilde{F}^{ij}=\frac{\partial F}{\partial u_{ij}}=\frac{\partial F}{\partial (U^t)_{rs}}\frac{\partial (U^t)_{rs}}{\partial u_{ij}}
=\lambda F^{rs}g_{rs}\delta_{ij}+F^{ij}.$$
By \eqref{eq4.1} and the Cauchy-Schwartz inequality
\begin{align}\label{eq4.4}
  \widetilde{F}^{ij}\frac{v_{i}v_{j}}{2v^{2}}%=&\,\frac{1}{2}\widetilde{F}^{ij}\left[\frac{\rho_{i}\rho_{j}}{\rho^{2}}+(\phi')^{2}u_{i}u_{j}+\phi' \frac{u_{i}\rho_{j}}{\rho}+\phi' \frac{u_{j}\rho_{i}}{\rho}\right]\\
 \leq\frac{3}{2}\widetilde{F}^{ij}\frac{\rho_{i}\rho_{j}}{\rho^{2}}+\frac{3}{4}(\phi')^{2} \widetilde{F}^{ij}u_{i}u_{j},
\end{align}
thus
\begin{align}\label{eq4.11}
\nonumber0\geq&\,\widetilde{F}^{ij}\left(\frac{\rho_{ij}}{\rho}-\frac{5}{2}\frac{\rho_{i}\rho_{j}}{\rho^{2}}\right)+\frac{1}{v}\sum_k\widetilde{F}^{ij}u_{ki}u_{kj}+\frac{1}{v}\sum_ku_{k}\widetilde{F}^{ij}u_{kij}+\phi'\widetilde{F}^{ij}u_{ij}\\
\nonumber&\,+\left[\phi^{''}-\frac{3(\phi')^{2}}{4}\right]\widetilde{F}^{ij}u_{i}u_{j}-\frac{1}{2v^2}\widetilde{F}^{ij}\left(\sum_ku_ku_{ki}\right)\left(\sum_lu_lu_{lj}\right)\\
%\nonumber=&\,\widetilde{F}^{ij}\left(\frac{\rho_{ij}}{\rho}-\frac{5}{2}\frac{\rho_{i}\rho_{j}}{\rho^{2}}\right)+\frac{1}{v}\sum_ku_{k}\widetilde{F}^{ij}u_{kij}+\phi'\widetilde{F}^{ij}u_{ij}\\
%\nonumber&\,+\left[\phi^{''}-\frac{3(\phi')^{2}}{4}\right]\widetilde{F}^{ij}u_{i}u_{j}+\frac{1}{v}\widetilde{F}^{ij}\left(\delta_{kl}-\frac{u_{k}u_{l}}{2v}\right)u_{ki}u_{lj}\\
\geq&\,\widetilde{F}^{ij}\left(\frac{\rho_{ij}}{\rho}-\frac{5}{2}\frac{\rho_{i}\rho_{j}}{\rho^{2}}\right)+\frac{1}{v}\sum_ku_{k}\widetilde{F}^{ij}u_{kij}+\phi'\widetilde{F}^{ij}u_{ij}\\
\nonumber&\,+\left[\phi^{''}-\frac{3(\phi')^{2}}{4}\right]\widetilde{F}^{ij}u_{i}u_{j}.
\end{align}
Recall
\begin{align*}
  (U^t)_{ij}=&\,\lambda(\Delta u)g_{ij}+u_{ij}+au_iu_j+b|\nabla u|^{2}g_{ij}+(B_g)_{ij},\\
  (U^t)_{ijk}=&\,\lambda(\Delta u)_{k}g_{ij}+u_{ijk}+a_{k}u_{i}u_{j}+au_{ik}u_{j}+au_{i}u_{jk} \\
  &\,+b_{k}|\nabla u|^{2}g_{ij}+2b\sum_pu_{p}u_{pk}g_{ij}+(B_g)_{ijk},
\end{align*}
then
\begin{align}\label{eq050602}
 \widetilde{ F}^{ij}u_{ij}%=&\left(\lambda F^{rs}g_{rs}\delta_{ij}+F^{ij}\right)u_{ij}\\
 %\nonumber=&\lambda F^{ij}g_{ij}\triangle u+F^{ij}u_{ij}\\
=F^{ij}(U^t)_{ij}-a(x)F^{ij}u_{i}u_{j}-b(x)|\nabla u|^{2}F^{ij}g_{ij}-F^{ij}(B_g)_{ij},
\end{align}
\begin{align}\label{eq050603}
  \widetilde{F}^{ij}u_{kij}
  %=&\,\widetilde{F}^{ij}u_{ijk}+\widetilde{F}^{ij}R_{pikj}u_{p}\\
 %=&\, \left(\lambda F^{rs}g_{rs}\delta_{ij}+F^{ij}\right)u_{ijk}+\widetilde{F}^{ij}R_{pikj}u_{p}\\
 %=&\,\lambda F^{ij}g_{ij}(\triangle u)_k+F^{ij}u_{ijk}+\widetilde{F}^{ij}R_{pikj}u_{p}\\
 % \nonumber=&\,F^{ij}[(U^t)_{ijk}-a_{k}u_{i}u_{j}-au_{ik}u_{j}-au_{i}u_{jk}-b_{k}|\nabla u|^{2}g_{ij}\\
  %&\,-2b\sum_pu_{p}u_{pk}g_{ij}-(B_g)_{ijk}]+\widetilde{F}^{ij}R_{pikj}u_{p}\\
  =&\,F^{ij}(U^t)_{ijk}-a_{k}F^{ij}u_{i}u_{j}-2aF^{ij}u_{i}u_{jk}-b_{k}|\nabla u|^{2}F^{ij}g_{ij}\\
  \nonumber&\,-2b\sum_pu_{p}u_{pk}F^{ij}g_{ij}-F^{ij}(B_g)_{ijk}+\widetilde{F}^{ij}R_{pikj}u_{p},
\end{align}
where $R_{pikj}$ is the Riemannnian curvature tensor of $(M,g)$.
Without loss of generality, we assume $\frac{1}{2}|\nabla u|^2\leq v\leq |\nabla u|^2$.
Substituting \eqref{eq050602} and \eqref{eq050603} into \eqref{eq4.11}, we have
\begin{align}\label{eq051801}
\nonumber0\geq&\,
v\widetilde{F}^{ij}\left(\rho_{ij}-\frac{5}{2}\frac{\rho_{i}\rho_{j}}{\rho}\right)+\rho v\left[\phi^{''}-\frac{3(\phi')^{2}}{4}\right]\widetilde{F}^{ij}u_{i}u_{j}\\
\nonumber&\,+\rho  u_kF^{ij}(U^t)_{ijk}+\phi'\rho vF^{ij}(U^t)_{ij}-\phi'\rho vF^{ij}(B_g)_{ij}-\rho u_kF^{ij}(B_g)_{ijk}\\
\nonumber&\,-\rho a_{k}F^{ij}u_{i}u_{j}u_k-\rho b_{k}u_k|\nabla u|^{2}F^{ij}g_{ij}+\rho u_ku_{p}\widetilde{F}^{ij}R_{pikj}\\
\nonumber&\, -2a\rho F^{ij}u_{i}u_ku_{jk}-2b\rho u_ku_{p}u_{pk}F^{ij}g_{ij}-\phi'\rho vaF^{ij}u_{i}u_{j}-\phi'\rho vb|\nabla u|^{2}F^{ij}g_{ij}\\
\geq&\,
\rho v\left[\phi^{''}-\frac{3(\phi')^{2}}{4}+a\phi'\right]\widetilde{F}^{ij}u_{i}u_{j}+(-\lambda a+b)\rho v \phi'F^{ij}g_{ij}|\nabla u|^{2}\\
\nonumber&\,+2avF^{ij}u_i\rho_j+2bvu_k\rho_k F^{ij}g_{ij}+v\widetilde{F}^{ij}\left(\rho_{ij}-\frac{5}{2}\frac{\rho_{i}\rho_{j}}{\rho}\right)\\
\nonumber&\,+\rho u_k F^{ij}(U^t)_{ijk}+\phi'\rho vF^{ij}(U^t)_{ij}-\rho a_{k}F^{ij}u_{i}u_{j}u_k-\rho b_{k}u_k|\nabla u|^{2}F^{ij}g_{ij}\\
\nonumber&\,-\phi'\rho vF^{ij}(B_g)_{ij}-\rho u_kF^{ij}(B_g)_{ijk}+\rho u_ku_{p}\widetilde{F}^{ij}R_{pikj}.
\end{align}
%From \eqref{eq4.1}
%\begin{align*}
%&-2a\rho F^{ij}u_{i}u_ku_{jk}-2b\rho u_ku_{p}u_{pk}F^{ij}g_{ij}-\phi'\rho vaF^{ij}u_{i}u_{j}-\phi'\rho vb|\nabla u|^{2}F^{ij}g_{ij}\\
%=&\,2a\rho v F^{ij}u_{i}\left(\frac{\rho_j}{\rho}+\phi'u_j\right)+2b\rho v F^{ij}g_{ij}u_k\left(\frac{\rho_k}{\rho}+\phi'u_k\right)\\
 %  &\,-\phi'\rho vaF^{ij}u_{i}u_{j}-\phi'\rho vb|\nabla u|^{2}F^{ij}g_{ij}\\
%=&\, 2avF^{ij}u_i\rho_j+2bvu_k\rho_k F^{ij}g_{ij}+a\rho v \phi'F^{ij}u_iu_j+b\rho v \phi'F^{ij}g_{ij}|\nabla u|^{2},
%\end{align*}
%then
%\begin{align}
%0\geq&\,
%\rho v\left[\phi^{''}-\frac{3(\phi')^{2}}{4}+a\phi'\right]\widetilde{F}^{ij}u_{i}u_{j}+(-\lambda a+b)\rho v \phi'F^{ij}g_{ij}|\nabla u|^{2}\\
%\nonumber&\,+2avF^{ij}u_i\rho_j+2bvu_k\rho_k F^{ij}g_{ij}+v\widetilde{F}^{ij}\left(\rho_{ij}-\frac{5}{2}\frac{\rho_{i}\rho_{j}}{\rho}\right)\\
%\nonumber&\,+\rho u_k F^{ij}(U^t)_{ijk}+\phi'\rho vF^{ij}(U^t)_{ij}-\rho a_{k}F^{ij}u_{i}u_{j}u_k-\rho b_{k}u_k|\nabla u|^{2}F^{ij}g_{ij}\\
%\nonumber&\,-\phi'\rho vF^{ij}(B_g)_{ij}-\rho u_kF^{ij}(B_g)_{ijk}+\rho u_ku_{p}\widetilde{F}^{ij}R_{pikj}.
%end{align}
By Lemma \ref{le051201}
\begin{align}\label{eq051802}
F^{ij}(U^t)_{ij}\leq 2e^{2u}\left[tf(x)+(1-t)\frac{\widetilde{S}_n(B_g)}{\widetilde{S}_{n-1}(B_g)}\right]\leq C.
\end{align}
%where $C$ is a positive constant depending on $\alpha, g, B, \|u\|_{C^0}, \|f\|_{C^0}$.
%where $C$ depends on $g, B$, $\|u\|_{C^0}$, $\|\alpha\|_{C^1}$, $\|f\|_{C^1}$.
Note that the following estimate holds
$$\left[\frac{\widetilde{S}_n}{\widetilde{S}_{n-1}}(B_g)\right]_k\geq -C|\nabla u|,$$
then
\begin{align}\label{eq051201}
F^{ij}(U^t)_{ijk}=&\,\left[e^{2u}\left(tf(x)+(1-t)\frac{\widetilde{S}_n}{\widetilde{S}_{n-1}}(B_g)\right)\right]_k\geq -C-Cv^{\frac{1}{2}}.
%=&\,2e^{2u}u_kf^t+te^{2u}f_k+(1-t)e^{2u}\left[\frac{\widetilde{S}_n(B_g)}{\widetilde{S}_{n-1}(B_g)}\right]_k\\
%\nonumber \geq&\, -C-Cv^{\frac{1}{2}},
\end{align}
Applying the idea of Guan-Wang \cite{GW03}, we assume
$|\nabla\rho|\leq C\rho^{1/2}$, $|\nabla^{2}\rho|\leq C$. Combining  \eqref{eq051801}-\eqref{eq051201}, we derive
\begin{align*}
0\geq&\,\rho v\left[\phi^{''}-\frac{3(\phi')^{2}}{4}+a\phi'\right]\widetilde{F}^{ij}u_{i}u_{j}+(-\lambda a+b)\rho v \phi'F^{ij}g_{ij}|\nabla u|^{2}\\
 &\,-C\rho v-C\rho v(-\phi')-C\rho v^{\frac{3}{2}}\sum_i F^{ii} -C\rho v^{\frac{3}{2}} F^{ij}g_{ij} -C\rho v(\lambda F^{rs}g_{rs}\delta_{ij}+F^{ij})\\
%\geq&\,\rho v\left[\phi^{''}-\frac{3(\phi')^{2}}{4}+a\phi'\right]\widetilde{F}^{ij}u_{i}u_{j}+(-\lambda a+b)\rho v \phi'F^{ij}g_{ij}|\nabla u|^{2}\\
%&\,-C\rho v-C\rho v(-\phi')-C v^{\frac{3}{2}}\sum_i F^{ii} \\
\geq&\,\lambda \rho v(-\phi')\left(\frac{N}{4 \varphi}-a\right)F^{ij}g_{ij}|\nabla u|^{2}+(\lambda a-b)\rho v(-\phi') F^{ij}g_{ij}|\nabla u|^{2}\\
&\,-C\rho v-C\rho v(-\phi')-C v^{\frac{3}{2}}\sum_i F^{ii}\\
\geq&\, \frac{1}{n}\rho v(-\phi')|\nabla u|^{2}\left[\lambda \left(\frac{N}{4 \varphi}-a\right)+\lambda a-b\right]-C\rho v-C\rho v(-\phi')-Cv^{\frac{3}{2}}.
\end{align*}
Now we suppose that $N$ is large enough such that $\frac{N\lambda}{8 \varphi}-b\geq 0$, then
$$0\geq  \frac{N}{8n\varphi}\lambda\rho (-\phi')v^2-CN\rho v-Cv^{\frac{3}{2}},$$
which implies $\rho(x_0)v(x_0)\leq C$. According to $Q(x)\leq Q(x_0)$,  gradient estimates have been derived.
\end{proof}

\subsection{Proof of $C^2$ eatimates}
In this subsection, we deduce  $C^2$ estimates for equation \eqref{eq050701}.

\begin{theorem}\label{thm051103}
Let $f(x)>0$ and $\alpha(x)\geq 0$ be smooth functions on $M$.
If $u\in C^{4}(B_{r})$ is a $(n-1)$-admissible solution of equation \eqref{eq050701} in $B_{r}\subset M$ with $\lambda>0$ and $B\in \Gamma_{n}$, then
\begin{equation*}
  \sup_{B_{r/2}}|\nabla^{2}u|\leq C,
\end{equation*}
where $C$ depends on $\lambda$, $g$, $\|\alpha\|_{C^{2}}$, $\|a\|_{C^{2}}$, $\|b\|_{C^{2}}$, $\|B\|_{C^{2}}$, $\|f\|_{C^2}$  and $\|u\|_{C^{1}}$.
\end{theorem}
\begin{proof}
From $g^{-1}\cdot U^t\in\Gamma_{2}\subset\Gamma_{n-1} $, we get $ |(U^t)_{ij}|\leq Ctr(U^t)$, thus $|u_{ij}|\leq C(\triangle u+1)$. Therefore, it suffices to establish an upper bound for $\Delta u$. Denote
\begin{equation*}
  P(x)=\rho(x)(\triangle u+N_1|\nabla u|^2),
\end{equation*}
where $\rho(x)\in C_{0}^{\infty}(B_{r})$ is the cutoff function defined in the proof of Theorem \ref{thm050701} and $N_1>1$ is a large constant. Assume $P$ reaches its maximum value at an interior point $x_1\in B_{r}$ and $U^t(u(x_1))$ is diagonal. From now on, all calculations are carried out at the point $x_1$.
\begin{align}
  0=P_{i}=&\,\rho_{i}(\triangle u+N_1|\nabla u|^2)+\rho[(\triangle u)_{i}+2N_1u_{k}u_{ki}],\label{eq07181}\\
  0\geq P_{ij}=&\,\rho_{ij}(\triangle u+N_1|\nabla u|^2)+\rho_{i}[(\triangle u)_{j}+2N_1u_{k}u_{kj}] \label{eq07182}\\
  \nonumber&+\rho_{j}[(\triangle u)_{i}+2N_1u_{k}u_{ki}]+\rho[(\triangle u)_{ij}+2N_1u_{ki}u_{kj}+2N_1u_{k}u_{kij}]\\
  \nonumber\geq&\,\rho(\triangle u)_{ij}+2N_1\rho u_{ki}u_{kj}+2N_1\rho u_{k}u_{kij}-CN_1(\triangle u+1).
\end{align}
Differentiate equation \eqref{eq050701}
\begin{align}
\nonumber F^{ij}(U^t)_{ijkk}=&\,(e^{2u}f^t)_{kk}-F^{ij,rs}(U^t)_{ijk}(U^t)_{rsk}\\
\nonumber\geq&\,4e^{2u}u_k^2f^t+2e^{2u}u_{kk}f^t+4e^{2u}u_k(f^t)_k+e^{2u}(f^t)_{kk}\\
\geq&\,-C|\nabla u|^2-C\triangle u+te^{2u}f_{kk}+(1-t)e^{2u}\left[\frac{\widetilde{S}_n(B_g)}{\widetilde{S}_{n-1}(B_g)}\right]_{kk}\\
\nonumber \geq&\, -C-C\triangle u,
\end{align}
where  we use  $\left[\frac{\widetilde{S}_n(B_g)}{\widetilde{S}_{n-1}(B_g)}\right]_{kk}\geq-C\triangle u$ in the last inequality.
%where $C$ depends on $\|u\|_{C^1}$, $\|\alpha\|_{C^2}$, $\|f\|_{C^2}$, $\|B\|_{C^2}$.
Recall
\begin{align*}
  (U^t)_{ijkk}=&\lambda(\triangle u)_{kk}g_{ij}+u_{ijkk}+a_{kk}u_{i}u_{j}+2a_{k}u_{ik}u_{j}+2a_{k}u_{i}u_{jk}\\
&+au_{ikk}u_{j}+2au_{ik}u_{jk}+au_{i}u_{jkk}+b_{kk}|\nabla u|^{2}g_{ij}\\
&+4b_{k}u_{p}u_{pk}g_{ij}+2bu_{pk}u_{pk}g_{ij}+2bu_{p}u_{pkk}g_{ij}+(B_g)_{ijkk},
\end{align*}
by the Ricci identity, we show
\begin{align}\label{eq07183}
   \widetilde{F}^{ij}u_{kkij}\geq &\, \widetilde{F}^{ij}u_{ijkk}-CF^{ij}(|\nabla^2u|+|\nabla u|)\\
   \nonumber=&\,F^{ij}(U^t)_{ijkk}-CF^{ij}(|\nabla^2u|+|\nabla u|)\\
  \nonumber &\,-F^{ij}[a_{kk}u_{i}u_{j}+4a_{k}u_{ik}u_{j}+2au_{ikk}u_{j}+2au_{ik}u_{jk}+b_{kk}|\nabla u|^{2}g_{ij}\\
  \nonumber &\,+4b_{k}u_{p}u_{pk}g_{ij}+2bu_{pk}u_{pk}g_{ij}+2bu_{p}u_{pkk}g_{ij}+(B_g)_{ijkk}]\\
  \nonumber\geq&\,-F^{ij}(2au_{ik}u_{jk}+2bu_{pk}u_{pk}g_{ij})-C\rho^{-\frac{1}{2}}F^{ij}g_{ij}(\triangle u+1)-C\triangle u.
\end{align}
Accord to \eqref{eq050603} and \eqref{eq051201}
\begin{align}
  \widetilde{F}^{ij}u_{kij}=&\,(e^{2u}f^t)_k-a_{k}F^{ij}u_{i}u_{j}-2aF^{ij}u_{i}u_{jk}-b_{k}|\nabla u|^{2}F^{ij}g_{ij}\\
  \nonumber&\,-2bu_{p}u_{pk}F^{ij}g_{ij}-F^{ij}(B_g)_{ijk}+\widetilde{F}^{ij}R_{pikj}u_{p}\\
  \nonumber\geq &\,-C(\triangle u+1)-CF^{ij}g_{ij}(\triangle u+1).
\end{align}
Then
\begin{equation*}
  \begin{aligned}
  0\geq&\,\widetilde{F}^{ij}P_{ij}\\
  \geq&\,-\rho F^{ij}(2au_{ik}u_{jk}+2bu_{pk}u_{pk}g_{ij})-C\rho^{\frac{1}{2}}F^{ij}g_{ij}(\triangle u+1)-C\rho \triangle u\\
  &+2N_1\rho \widetilde{F}^{ij}u_{ki}u_{kj}+2N_1\rho \widetilde{F}^{ij}u_{k}u_{kij}-CN_1\widetilde{F}^{ij}(\triangle u+1)\\
  \geq&\,2N_1\rho (\lambda F^{rs}g_{rs}\delta_{ij})u_{ki}u_{kj}-\rho F^{ij}(2au_{ik}u_{jk}+2bu_{pk}u_{pk}g_{ij})\\
  &-CN_1F^{ij}g_{ij}(\triangle u+1)-CN_1\triangle u-CN_1\widetilde{F}^{ij}(\triangle u+1)\\
  \geq&\,\frac{2N_1}{n}\rho \lambda F^{ij}g_{ij}(\triangle u)^2-\rho F^{ij}g_{ij}[2|a|(\triangle u)^2+2|b|(\triangle u)^2]\\
  &-CN_1F^{ij}g_{ij}(\triangle u+1)-CN_1\triangle u\\
  \geq&\,F^{ij}g_{ij}\left[\frac{2}{n}N_1\rho \lambda(\triangle u)^2-\rho(2|a|+2|b|)(\triangle u)^2-CN_1(\triangle u+1)\right]-CN_1\triangle u,
  \end{aligned}
\end{equation*}
where $C$ is a positive constant depending on $\lambda$, $g$, $\|\alpha\|_{C^{2}}$, $\|a\|_{C^{2}}$, $\|b\|_{C^{2}}$, $\|B\|_{C^{2}}$, $\|f\|_{C^2}$ and $\|u\|_{C^{1}}$.
If $$\frac{2}{n}N_1\rho \lambda(\triangle u)^2-\rho(2|a|+2|b|)(\triangle u)^2-CN_1(\triangle u+1)\leq 0,$$
then $\triangle u\leq C$ by choosing $N_1$ sufficiently large. Otherwise
\begin{align*}
  0\geq&\, \frac{1}{n}\left[\frac{2}{n}N_1\rho \lambda(\triangle u)^2-\rho(2|a|+2|b|)(\triangle u)^2-CN_1(\triangle u+1)\right]-CN_1\triangle u\\
  \geq&\,\frac{N_1}{n^2}\rho \lambda(\triangle u)^2-CN_1\triangle u.
\end{align*}
Therefore, we complete the proof.
\end{proof}

\subsection{Proof of Theorem \ref{thm051104}}
In this section, we utilize the degree theory for fully nonlinear elliptic equations developed in \cite{Li89}. Given that the proof is analogous to the method established in \cite{An, Jin, LS05}, we present a concise sketch of the continuity path and the verification of the non-zero degree, which ensures the solvability of our target equation.
%In this section, we apply the degree theory for nonlinear elliptic equations developed in \cite{Li89} to establish Theorem \ref{thm051104}. Since the argument is analogous to those in \cite{An, Jin, LS05}, we provide only a brief sketch here.

The a priori estimates derived in Theorems \ref{thm051101}, \ref{thm050701}, \ref{thm051103} ensure that equation \eqref{eq050701} is uniformly elliptic. In view of the Evans-Krylov theory \cite{Eva82, Kry83} and standard Schauder estimates, it follows that
\begin{eqnarray}\label{C2+}
|u|_{C^{4,\delta}}\leq C, \quad 0< \delta<1.
\end{eqnarray}
Next we consider a family of the mapping for $t \in [0,1]$
$$\widetilde{G}(.; t): C_{0}^{4,\delta}(M)\rightarrow
C^{2,\delta}(M),$$
which is defined by
$$\widetilde{G}(u;t)=F(U^t)-e^{2u}f^t(x).$$
And $C_{0}^{4,\delta}(M)=\{u \in C^{4,\delta}(M), \, \lambda(g^{-1}\cdot U^t) \in \Gamma_{n-1}\}$ is an open set of $C^{4,\delta}(M)$.
Let
$$\mathcal{O}_R=\{\rho \in C_{0}^{4,\delta}(M): \|u\|_{C^{4,\delta}}<R\},$$
which is also an open subset of $C^{4,\delta}(M)$.
Moreover, for a sufficiently large $R$, the a priori estimates established in \eqref{C2+} ensure that  $\widetilde{G}(u; t)=0$ possesses no solution on $\partial \mathcal{O}_R$. Consequently, the degree $\deg(\widetilde{G}(.; t), \mathcal{O}_R, 0)$ is well-defined for $0\leq t\leq 1$.

When $t=0$, Lemma \ref{le050901} shows that $u=0$ is the unique solution to equation \eqref{eq050701}.  Let
$u(x,s)$ be the variation of $u = 0$ such that at $u_s'=\phi$ at $s =0$. A straightforward calculation shows that
$$\delta_{\phi}\widetilde{G}(0; 0)= a^{ij}\phi_{ij}+b^i \phi_i-2\frac{\sigma_n(g)+\alpha \sigma_{n-1}(g)}{\sigma_{n-1}(g)
+\alpha \sigma_{n-2}(g)}\phi,$$
where $a^{ij}$ is a positive definite matrix and  $\delta_{\phi}\widetilde{G}(0; 0)$ is the linearized operator of $\widetilde{G}$ at $u=0$. Hence, $\delta_{\phi}\widetilde{G}(0; 0)$ is an in invertible operator. Therefore
\begin{eqnarray*}
\deg(\widetilde{G}(.; 1), \mathcal{O}_R; 0)=\deg(\widetilde{G}(.; 0), \mathcal{O}_R, 0)=\pm
1.
\end{eqnarray*}
So we obtain a solution at $t=1$. This completes the proof of Theorem \ref{thm051104}.

\vskip 0.1cm {\bf Acknowledgments}
\quad Authors would like to express sincere gratitude to Prof. Yan He for giving us many professional and helpful suggestions in our work.
\vskip 0.1cm
{\bf Data Availability Statement}\quad This paper doesn't incorporate any data.
\vskip 0.1cm
{\bf Conflict of Interest}\quad The authors declare no conflict of interest.

%%%%%%%%%%%%%%%%%%%%%%%%%%%%%%%%%%%%%%%%%%%%%%%%%%%%%%%%%%%%%%%%%%%%%%%%%%%%%%%%%%%%%%%%%%%%%%%%%%%%%%%%%%%%%%%%%%%%%%%%%%%%%%%%%%%%%%%%%%%%%%%%%%%%%%%%%%%%%%%%%%%%

\end{document}